\documentclass{paper}
\usepackage[utf8]{inputenc}
\usepackage{float}
\usepackage{amsmath}
\usepackage{amsthm}
\usepackage{amssymb}
\usepackage{graphicx}
\usepackage[numbers]{natbib}
\usepackage[unicode=true,
 bookmarks=false,
 breaklinks=false,pdfborder={0 0 1},backref=section,colorlinks=false]
 {hyperref}

\makeatletter
\theoremstyle{plain}
\newtheorem{thm}{\protect\theoremname}
\theoremstyle{plain}
\newtheorem{prop}[thm]{\protect\propositionname}
\ifx\proof\undefined
\newenvironment{proof}[1][\protect\proofname]{\par
	\normalfont\topsep6\p@\@plus6\p@\relax
	\trivlist
	\itemindent\parindent
	\item[\hskip\labelsep\scshape #1]\ignorespaces
}{%
	\endtrivlist\@endpefalse
}
\providecommand{\proofname}{Proof}
\fi

\usepackage{stmaryrd}

\makeatother

\providecommand{\propositionname}{Proposition}
\providecommand{\theoremname}{Theorem}

\begin{document}
\title{Optimal separator for an ellipse\\
Application to localization}
\author{Luc Jaulin}
\institution{Lab-Sticc, ENSTA-Bretagne}

\maketitle
\textbf{Abstract}. This paper proposes a minimal contractor and a
minimal separator for an ellipse in the plane. The task is facilitated
using actions induced by the hyperoctahedral group of symmetries.
An application related to the localization of an object using multiple
sonars is proposed. 

\section{Introduction}

Consider the quadratic function 
\begin{equation}
f(\mathbf{q},\mathbf{x})=q_{0}+q_{1}x_{1}+q_{2}x_{2}+q_{3}x_{1}^{2}+q_{4}x_{1}x_{2}+q_{5}x_{2}^{2}\label{eq:fqdex}
\end{equation}
where $\mathbf{q}=(q_{0},\dots,q_{5})$ is the parameter vector and
$\mathbf{x}=(x_{1},x_{2})$ is the vector of variables. Equivalently,
we can write the function in a matrix form: 
\begin{equation}
f(\mathbf{q},\mathbf{x})=\mathbf{x}^{\text{T}}\cdot\left(\begin{array}{cc}
q_{3} & \frac{1}{2}q_{4}\\
\frac{1}{2}q_{4} & q_{5}
\end{array}\right)\cdot\mathbf{x}+(q_{1}\,\,\,\,q_{2})\cdot\mathbf{x}+q_{0}.
\end{equation}

The zeros of this quadratic function is, in general, a conic section
(a circle or other ellipse, a parabola, or a hyperbola). Define the
set
\begin{equation}
\mathbb{X}=\left\{ (x_{1},x_{2}|f(\mathbf{q},\mathbf{x})\leq0\right\} .
\end{equation}
We will assume here that the square matrix involved in the matrix
form has positive eigen values. In this case $\mathbb{X}$ is an ellipse.
In this paper, we propose an interval-based method \citep{Moore79}
to generate an optimal separator \citep{DesrochersEAAI2014} for the
set $\mathbb{X}$. This separator will be used to generate an inner
and an outer approximations for $\mathbb{X}$. As an application,
we will consider the problem of the localization of an object using
3 sonars.

This paper is organized as follows. Section \ref{sec:Symmetries}
introduces the notion of symmetries that will be used in the construction
of the separators. Section \ref{sec:Separator} builds the separator
for the ellipse. Section \ref{sec:Application} illustrates the use
of the separator to approximate the set of position for an object
using three sonars. Section \ref{sec:Conclusion} concludes the paper.

\section{Symmetries\label{sec:Symmetries}}

Define an equation of the form
\[
f(\mathbf{q},\mathbf{x})=0.
\]

Two transformations $\varepsilon$ and $\sigma$ are conjugate with
respect to $f$ if 
\[
f(\varepsilon(\mathbf{q}),\sigma(\mathbf{x}))=0\Leftrightarrow f(\mathbf{q},\mathbf{x})=0.
\]
Transformations that will be consider are limited to the \emph{hyperoctahedral
group} $B_{n}$ \citep{Coxeter99} which is the group of symmetries
of the hypercube $[-1,1]^{n}$ of $\mathbb{R}^{n}$. The group $B_{n}$
corresponds to the group of $n\times n$ orthogonal matrices whose
entries are integers. Each line and each column of a matrix should
contain one and only one non zero entry which should be either $1$
or $-1$. Figure \ref{fig:bn_repr1.png} shows different notations
usually considered to represent a symmetry $\sigma$ of $B_{5}$.
We will prefer the Cauchy one line notation \citep{Wussing07} which
is shorter. We should understand the symmetry $\sigma$ of the figure
as the function:
\begin{equation}
\sigma(x_{1},x_{2},x_{3},x_{4},x_{5})=(-x_{2},x_{1},x_{5},-x_{4},x_{3}).
\end{equation}

\begin{figure}[H]
\begin{centering}
\includegraphics[width=7cm]{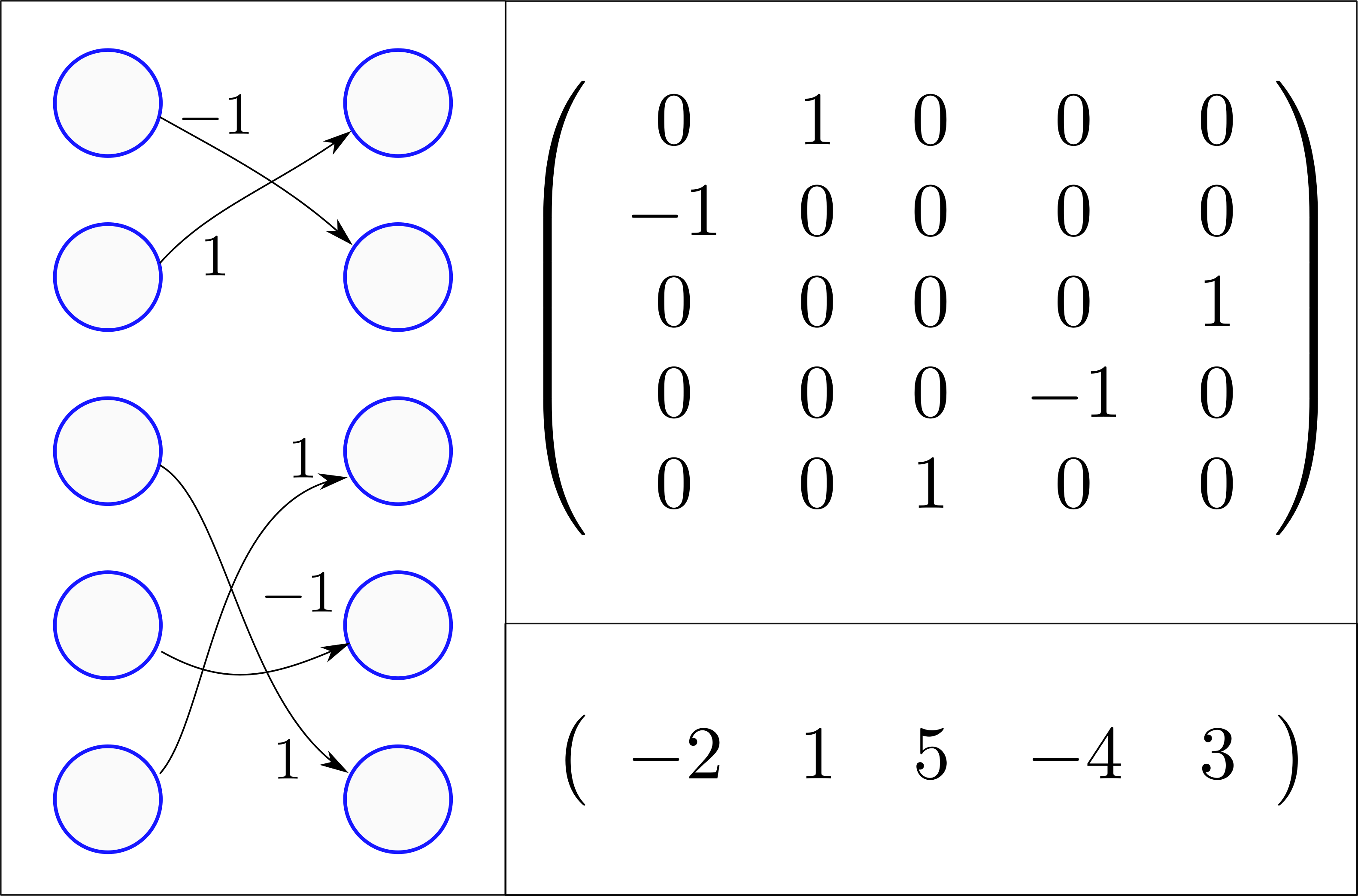}
\par\end{centering}
\caption{Different representations of an element $\sigma$ of $B_{5}$. Left:
graph; Top right: Matrix notation; Bottom right: Cauchy one line notation}
\label{fig:bn_repr1.png}
\end{figure}
Even if the matrix representation looks more intuitive, for efficiency
reasons, we use the Cauchy one line representation to compose the
symmetries. Let us consider again the function

\begin{equation}
f(x_{1},x_{2})\overset{(\ref{eq:fqdex})}{=}q_{0}+q_{1}x_{1}+q_{2}x_{2}+q_{3}x_{1}^{2}+q_{4}x_{1}x_{2}+q_{5}x_{2}^{2}.
\end{equation}
Take the symmetry
\[
x_{1}\rightarrow\varepsilon_{1}x_{1};x_{2}\rightarrow\varepsilon_{2}x_{1}
\]
where $\varepsilon_{i}\in\{-1,1\}$. With the Cauchy notation, this
transformation is denoted by $\varepsilon=(\varepsilon_{1},2\varepsilon_{2})$.
We have
\begin{equation}
\begin{array}{ccl}
f(\varepsilon\mathbf{x}) & = & q_{0}+q_{1}\varepsilon_{1}x_{1}+q_{2}\varepsilon_{2}x_{2}+q_{3}x_{1}^{2}+q_{4}\varepsilon_{1}\varepsilon_{2}x_{1}x_{2}+q_{5}x_{2}^{2}\end{array}
\end{equation}
As a consequence, for each symmetry $\varepsilon=(\varepsilon_{1},2\varepsilon_{2})$,
the pair
\begin{equation}
((\varepsilon_{1},2\varepsilon_{2}),(1,2\varepsilon_{1},3\varepsilon_{2},4,5\varepsilon_{1}\varepsilon_{2},6))
\end{equation}
is conjugate. We thus get the choice function $\psi$ \citep{jaulin:quotient:2023}:
\begin{equation}
\psi(\varepsilon_{1},2\varepsilon_{2})=(1,2\varepsilon_{1},3\varepsilon_{2},4,5\varepsilon_{1}\varepsilon_{2},6)\label{eq:psi}
\end{equation}
Given a symmetry $\varepsilon$, this choice function allows us to
get a symmetry $\sigma$ such that $(\varepsilon,\sigma)$ is a conjugate
pair. 

\section{Separator for the ellipse\label{sec:Separator}}

This section proposes an optimal separator for an ellipse. This operator
will be used later by a paver to compute boxes that are completely
inside or outside the ellipse. 

\subsection{Cardinal points}

We define the cardinal points as the points $(x_{1},x_{2})$ which
satisfies 
\[
\left\{ \begin{array}{c}
f(\mathbf{x})=0\\
\frac{\partial f}{\partial x_{1}}(\mathbf{x})=0\,\text{or}\,\frac{\partial f}{\partial x_{2}}(\mathbf{x})=0
\end{array}\right.
\]
Generically, there exist four cardinal points. The cardinal point,
painted red in Figure \ref{fig:ellipse0}, at the top (left, bottom,
right) is the North (West, South, East, respectively). 

\subsection{Contractor for the positive quadrant\label{subsec:positive:quadrant}}

The part of the ellipse between the North and the East is called the
\emph{positive arc} and is painted blue in Figure \ref{fig:ellipse0}.
The smallest box which encloses this arc is the \emph{positive quadrant}.

\begin{figure}[H]
\begin{centering}
\includegraphics[width=9cm]{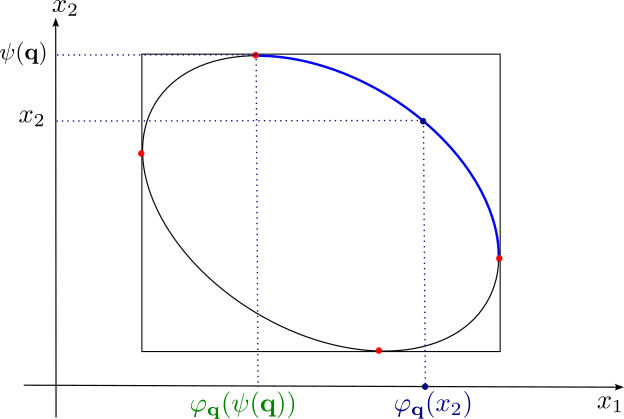}
\par\end{centering}
\caption{Positive arc and the corresponding function $\varphi_{\mathbf{q}}(x_{2})$ }
\label{fig:ellipse0}
\end{figure}

\begin{prop}
\label{prop:phiq}Take a point $x=(x_{1},x_{2})$ such that $f(\mathbf{x})=0$
of the positive quadrant. We have
\begin{equation}
\begin{array}{ccl}
x_{1} & = & \varphi_{\mathbf{q}}(x_{2})\\
 & = & \frac{-(q_{1}+q_{4}x_{2})\pm\sqrt{(q_{1}+q_{4}x_{2})^{2}-4q_{1}(q_{0}+q_{2}x_{2}+q_{5}x_{2}^{2})}}{2q_{3}}
\end{array}\label{eq:x1}
\end{equation}
The largest feasible $x_{2}$ is
\begin{equation}
\begin{array}{ccc}
x_{2} & = & \psi(\mathbf{q})\\
 & = & \frac{-4q_{3}q_{2}+2q_{1}q_{4}+\sqrt{(4q_{3}q_{2}-2q_{1}q_{4})^{2}-4(4q_{3}q_{5}-q_{4}^{2})(4q_{3}q_{0}-q_{1}^{2})}}{8q_{3}q_{5}-2q_{4}^{2}}
\end{array}\label{eq:X2}
\end{equation}
The North has the coordinates $\left(\varphi_{\mathbf{q}}(\psi(\mathbf{q})),\psi(\mathbf{q})\right)$.
\end{prop}
\begin{proof}
Assume that $x_{2}$ is known. Let us compute the possible value for
$x_{1}$, if it exists. Since
\begin{equation}
f(x_{1},x_{2})\overset{(\ref{eq:fqdex})}{=}q_{3}x_{1}^{2}+\left(q_{1}+q_{4}x_{2}\right)x_{1}+q_{2}x_{2}+q_{0}+q_{5}x_{2}^{2},
\end{equation}
we get the following discriminant:
\begin{equation}
\Delta_{1}=b_{1}^{2}-4a_{1}c_{1}
\end{equation}
where
\begin{equation}
a_{1}=q_{3},\,b_{1}=q_{1}+q_{4}x_{2},\,c_{1}=q_{0}+q_{2}x_{2}+q_{5}x_{2}^{2}
\end{equation}
The two solutions are
\begin{equation}
x_{1}=\frac{-b_{1}\pm\sqrt{\Delta_{1}}}{2a_{1}}.
\end{equation}
We have thus proved corresponds (\ref{eq:x1}) .

A value for $x_{2}$ yields a feasible $x_{1}$ if $\Delta_{1}\geq0$,
\emph{i.e.},
\[
\begin{array}{clc}
 & b_{1}^{2}-4a_{1}c_{1} & \geq0\\
\Leftrightarrow & -(q_{1}+q_{4}x_{2})^{2}+4q_{3}(q_{0}+q_{2}x_{2}+q_{5}x_{2}^{2}) & \geq0\\
\Leftrightarrow & (4q_{3}q_{5}-q_{4}^{2})x_{2}^{2}+(4q_{3}q_{2}-2q_{1}q_{4})x_{2}+4q_{3}q_{0}-q_{1}^{2} & \leq0
\end{array}
\]
which is quadratic in $x_{2}.$ The discriminant is
\begin{equation}
\Delta_{2}=b_{2}^{2}-4a_{2}c_{2}
\end{equation}
where
\begin{equation}
\begin{array}{ccc}
a_{2} & = & 4q_{3}q_{5}-q_{4}^{2}\\
b_{2} & = & 4q_{3}q_{2}-2q_{1}q_{4}\\
c_{2} & = & 4q_{3}q_{0}-q_{1}^{2}
\end{array}
\end{equation}
We thus get \ref{eq:X2}.
\end{proof}
\begin{prop}
\label{prop:phi:psi}Take a point $x=(x_{1},x_{2})$ such that $f(\mathbf{x})=0$
of the positive quadrant. We have
\begin{equation}
x_{2}=\varphi_{\sigma(\mathbf{q})}(x_{1})\label{eq:x2}
\end{equation}
where $\sigma=(1,3,2,6,5,4])$. The largest feasible $x_{1}$ is
\begin{equation}
\begin{array}{ccc}
x_{1} & = & \psi(\sigma(\mathbf{q})).\end{array}\label{eq:X1}
\end{equation}
The East has the coordinates $\left(\psi(\sigma(\mathbf{q})),\varphi_{\sigma(\mathbf{q})}(\psi(\sigma(\mathbf{q})))\right)$.
\end{prop}
\begin{proof}
The symmetry which permutes $x_{1},x_{2}$ is $\sigma=(1,3,2,6,5,4])$.
Indeed:
\begin{equation}
\begin{array}{ccc}
f(x_{2},x_{1}) & \overset{(\ref{eq:fqdex})}{=} & q_{0}+q_{1}x_{2}+q_{2}x_{1}+q_{3}x_{2}^{2}+q_{4}x_{1}x_{2}+q_{5}x_{1}^{2}\\
 & = & q_{0}+q_{2}x_{1}+q_{1}x_{2}+q_{5}x_{1}^{2}+q_{4}x_{1}x_{2}+q_{3}x_{2}^{2}
\end{array}
\end{equation}

After application of the symmetry, Proposition \ref{prop:phi:psi}
falls in the conditions of Proposition \ref{prop:phiq} (see Figure
\ref{fig:ellipse0cardinalreverse}).
\end{proof}
\begin{figure}[H]
\begin{centering}
\includegraphics[width=9cm]{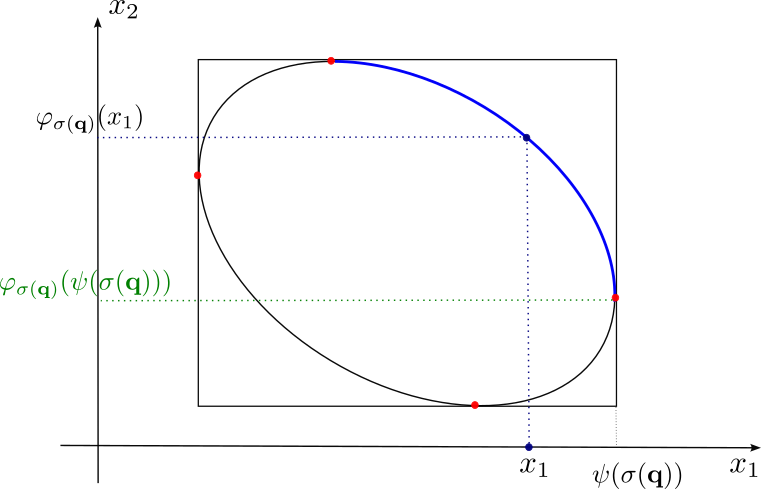}
\par\end{centering}
\caption{Positive arc and the corresponding function $\varphi_{\sigma(\mathbf{q})}$
after permutation}
\label{fig:ellipse0cardinalreverse}
\end{figure}

\begin{prop}
The smallest box which contains the North and the East is
\begin{equation}
[\mathbf{a}]=[\varphi_{\mathbf{q}}(\psi(\mathbf{q})),\psi(\sigma(\mathbf{q}))]\times[\varphi_{\sigma(\mathbf{q})}(\psi(\sigma(\mathbf{q}))),\psi(\mathbf{q})].
\end{equation}
\end{prop}
\begin{proof}
The result can be read directly from Figure \ref{fig:ellipse0_hull}.
\end{proof}
\begin{figure}[H]
\begin{centering}
\includegraphics[width=9cm]{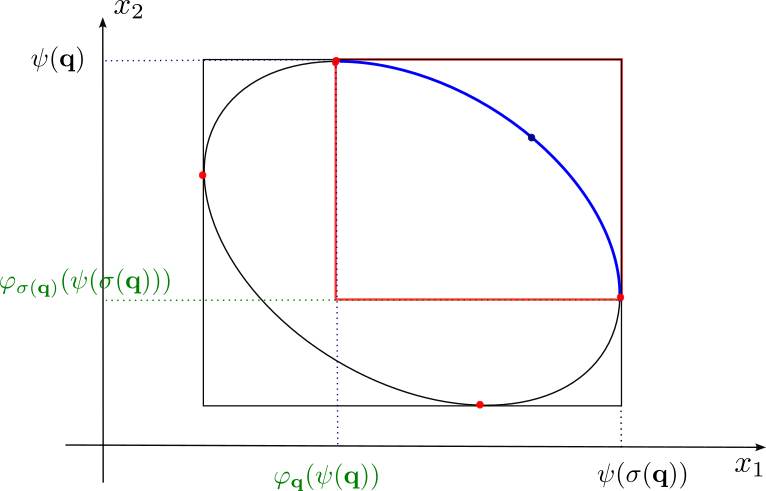}
\par\end{centering}
\caption{Smallest box which encloses the North and the East}
\label{fig:ellipse0_hull}
\end{figure}

\begin{prop}
The minimal contractor associated to the positive ellipse is 
\begin{equation}
C_{0}^{\mathbf{q}}([\mathbf{x}])=[\mathbf{x}]\cap\left(\begin{array}{c}
[\varphi_{\mathbf{q}}(b_{2}^{+}),\varphi_{\mathbf{q}}(b_{2}^{-})]\\{}
[\varphi_{\sigma(\mathbf{q})}(b_{1}^{+}),\varphi_{\sigma(\mathbf{q})}(b_{1}^{-})]
\end{array}\right)
\end{equation}
with $[\mathbf{b}]=[\mathbf{x}]\cap[\mathbf{a}].$
\end{prop}
\begin{proof}
This is a direct consequence of the monotonicity of the partial function
$\varphi_{\mathbf{q}}$.
\end{proof}
If we apply this contractor in a paver with $\mathbf{q}=(-5,1,1,3,1,2)$,
we get Figure \ref{fig:sivia_ell1.png}.
\begin{figure}[H]
\begin{centering}
\includegraphics[width=8cm]{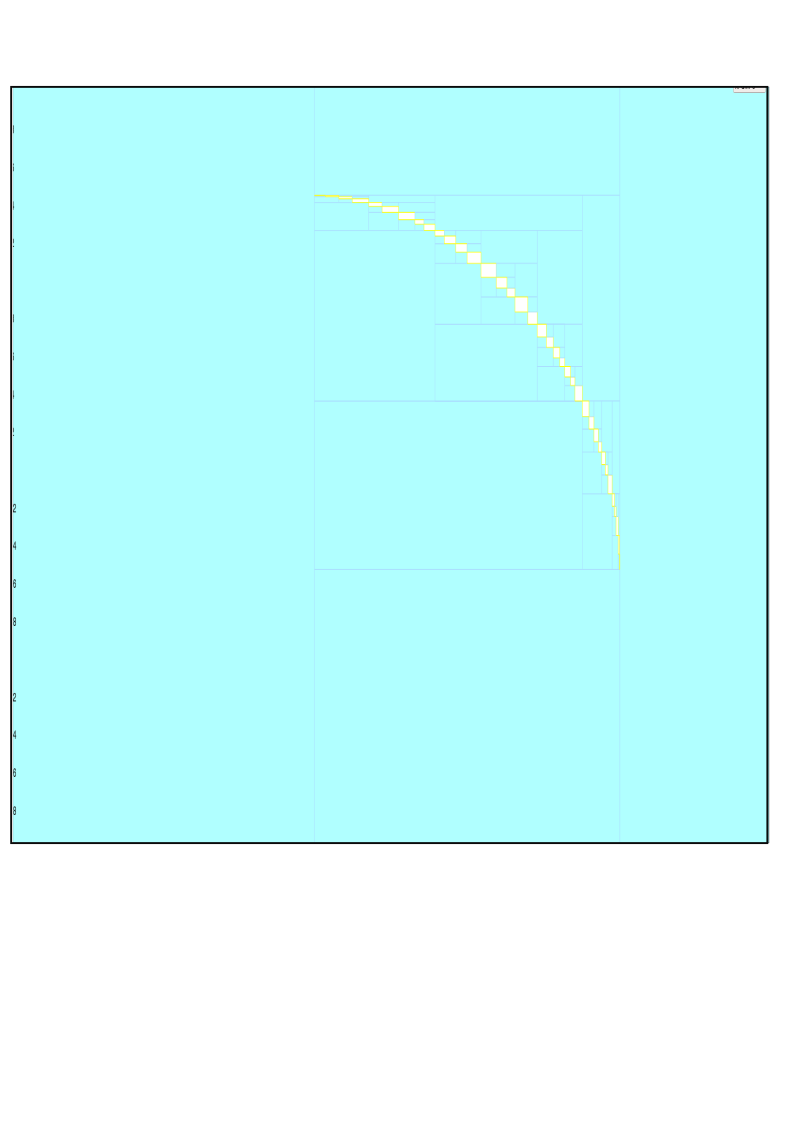}
\par\end{centering}
\caption{Illustration of the application of the contractor for the ellipse
on the positive quadrant}
\label{fig:sivia_ell1.png}
\end{figure}

\subsection{Contractor for the ellipse boundary}

Subsection \ref{subsec:positive:quadrant} has shown how to build
a contractor for the North-East quadrant of the ellipse. Recall that
$C_{0}^{\mathbf{q}}([\mathbf{x}])$ contracts the box $[\mathbf{x}]$
with respect to the positive quadrant of the ellipse. It depend on
the parameter vector $\mathbf{q}$ of the ellipse. Using the notion
of contractor action \citep{jaulin:IA:hyperoctohedral}, we show how
we can extend this contractor $C_{0}^{\mathbf{q}}$ to other quadrants.
We recall that the action of a symmetry $\varepsilon$ to the contractor
$C$ is defined by
\[
\varepsilon\bullet C([\mathbf{x}])=\varepsilon\circ C\circ\varepsilon^{-1}([\mathbf{x}]).
\]
This means that $\varepsilon\bullet C$ is a contractor that has been
built from the contractor $C$ as follows:
\begin{itemize}
\item Apply to the box $[\mathbf{x}]$ the symmetry $\varepsilon^{-1}$
\item Apply the contractor $C$
\item Apply to the resulting box $C\circ\varepsilon^{-1}([\mathbf{x}])$
the symmetry $\varepsilon$.
\end{itemize}
If we consider the pair $(\varepsilon,\sigma)$ conjugate with respect
to the ellipse, the contractor $\varepsilon\bullet C_{0}^{\sigma(\mathbf{q})}$
is associated to another quadrant of the ellipse. The selection of
the symmetries $(\varepsilon,\sigma)$ to be selected is made using
the choice function (\ref{eq:psi}). In the ellipse case, we clearly
understand geometrically that 4 symmetries are needed since the ellipse
has 4 quadrants (North-East, North-West, South-West, South-East).
These symmetries can be computed automatically as shown in \citep{jaulin:IA:hyperoctohedral}.

The contractor for the ellipse boundary is thus
\begin{equation}
[\mathbf{x}]\mapsto\bigcup_{^{\varepsilon\in\{(1,2),(1,-2),(-1,2),(-1,-2)\}}}(\varepsilon\bullet C_{0}^{\psi_{\varepsilon}(\mathbf{q})})([\mathbf{x}]).
\end{equation}
The application of this contractor is illustrated by Figure \ref{fig:sivia_ell1}.

\begin{figure}[H]
\begin{centering}
\includegraphics[width=8cm]{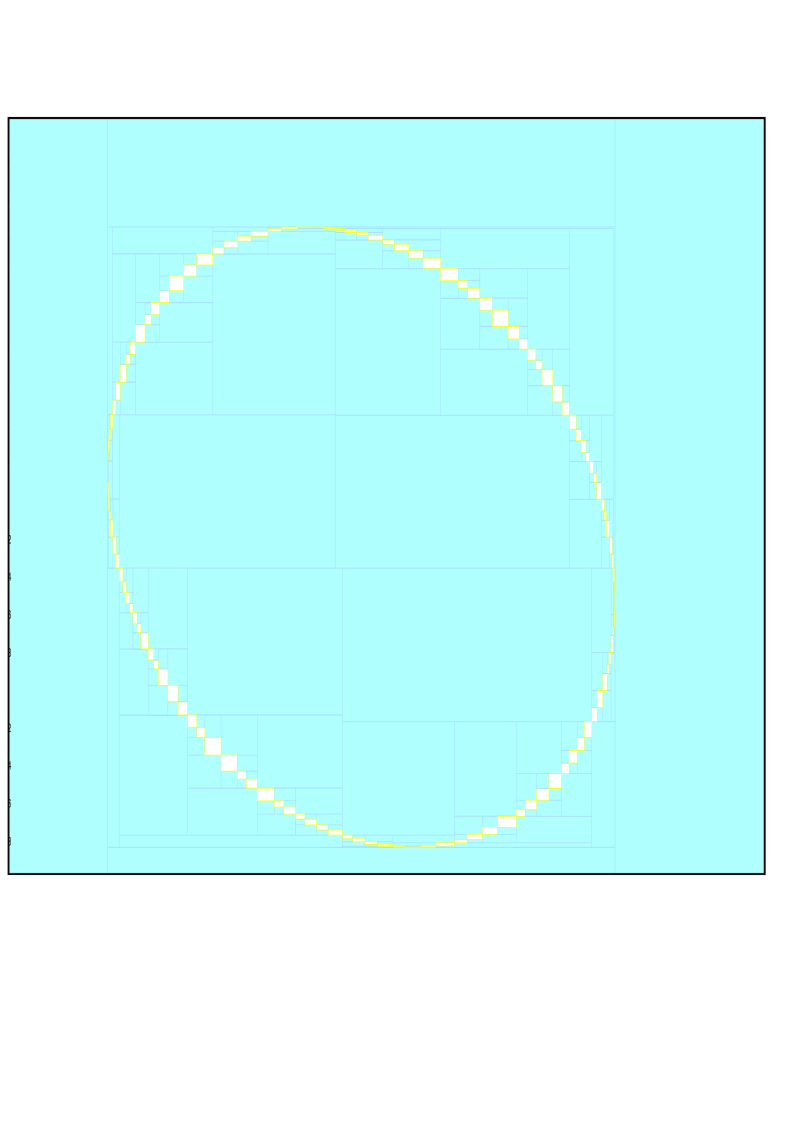}
\par\end{centering}
\caption{Illustration of the application of the contractor for the ellipse
on all quadrants}
\label{fig:sivia_ell1}
\end{figure}

\subsection{Separator for the ellipse}

From a contractor on the boundary of a set $\mathbb{X}$ and a test
for $\mathbb{X}$, we can obtain a separator. As a consequence, we
can get an inner and an outer approximations for $\mathbb{X}$ as
illustrated by Figure \ref{fig:sivia_ell2}.

\begin{figure}[H]
\begin{centering}
\includegraphics[width=8cm]{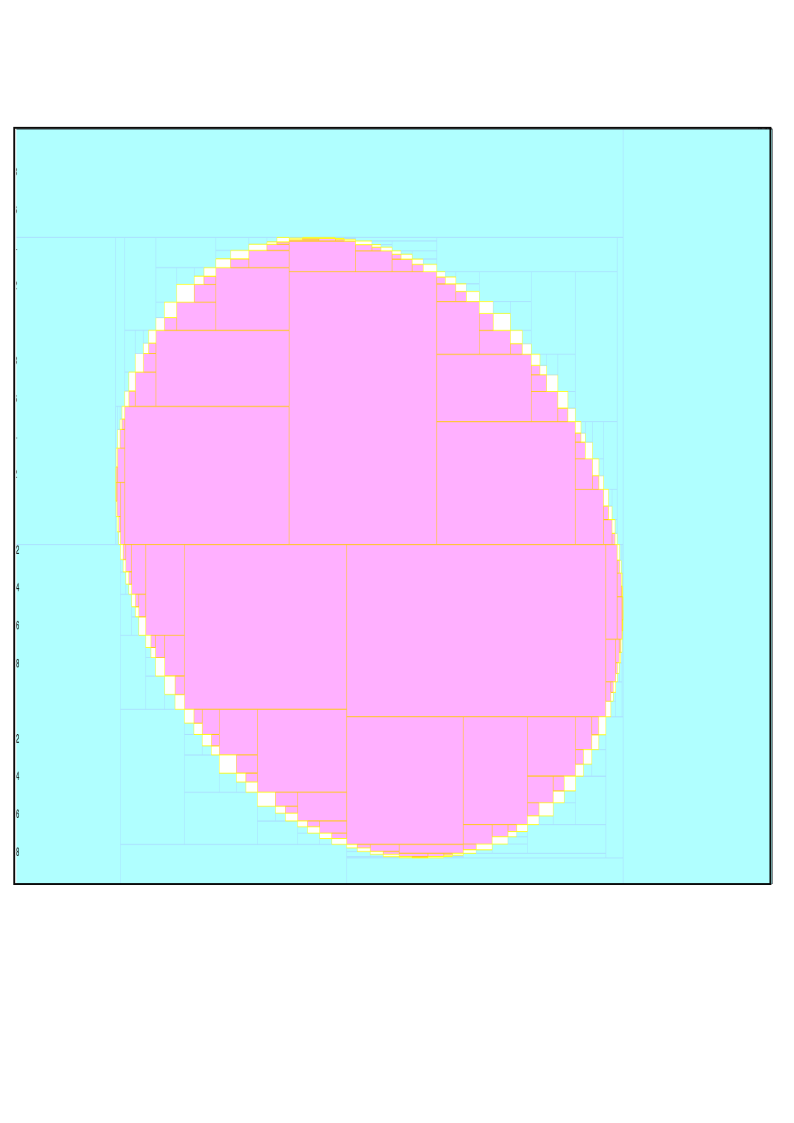}
\par\end{centering}
\caption{Illustration of the application of the separator for the ellipse on
all quadrants}
\label{fig:sivia_ell2}
\end{figure}

If we compare with a classical forward-backward contractor \citep{Ben99}
(see \ref{fig:sivia_ell3}) of other contractors such as \citep{DBLP:Araya}
our contractor yields a more accurate approximation.

\begin{figure}[H]
\begin{centering}
\includegraphics[width=8cm]{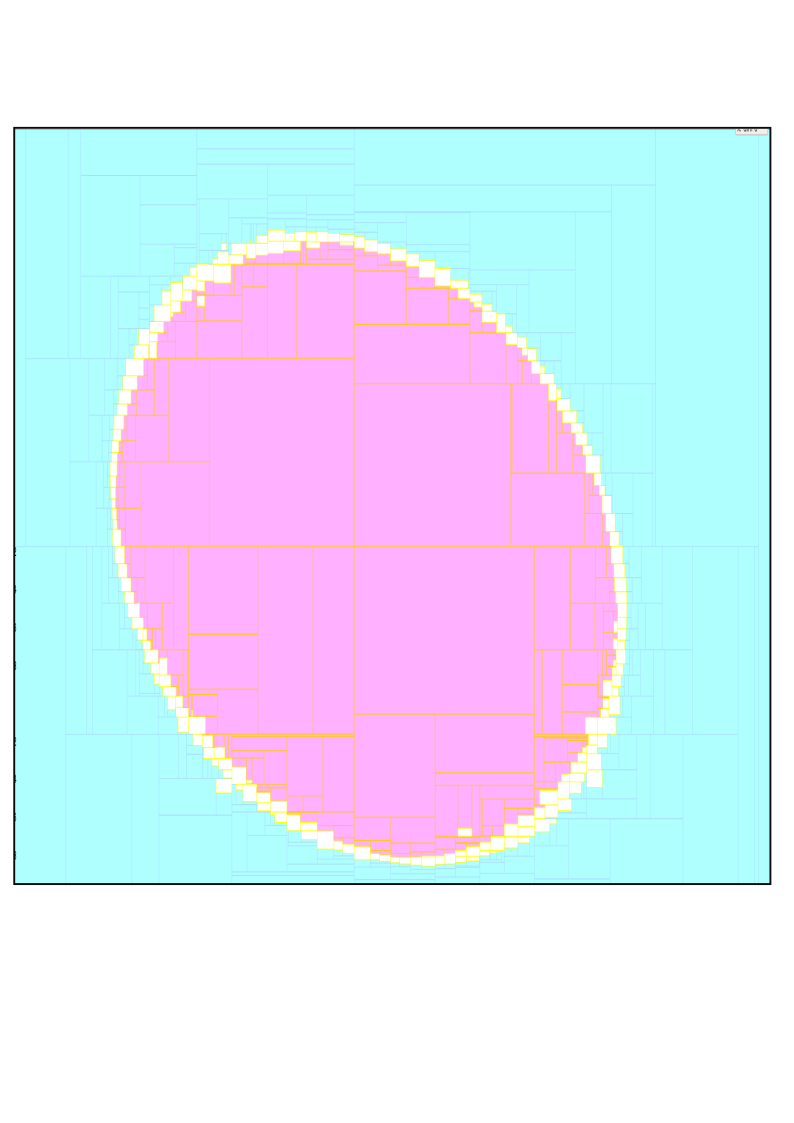}
\par\end{centering}
\caption{Classical interval contractors require a larger number of bisections }
\label{fig:sivia_ell3}
\end{figure}
\textbf{Remark}. We have assumed that we had no uncertainties on $\mathbf{q}$.
In case of interval uncertainty, the set to be characterized becomes
\begin{equation}
\mathbb{X}=\{\mathbf{x}|\exists\mathbf{q}\in[\mathbf{q}],q_{0}+q_{1}x_{1}+q_{2}x_{2}+q_{3}x_{1}^{2}+q_{4}x_{1}x_{2}+q_{5}x_{2}^{2}\leq0\}.
\end{equation}

The resolution is still possible as shown in \citep{jaulin:quotient:2023}.

\section{Application\label{sec:Application}}

Interval methods have been used for localization of robots for several
decades \citep{JaulinBook01}\citep{robloc}\citep{colle_galerne13}\citep{Drevelle:Bonnifait:09}.
This section proposes to deal with a specific localization problem
where the sum of distances are measured. 

\subsection{Ellipse}
\begin{prop}
\label{prop:ellipse:foci}Consider two points $\mathbf{a},\mathbf{b}$
of the plane. The set $\mathbb{X}$ of all points such that 
\begin{equation}
\|\mathbf{x}-\mathbf{a}\|+\|\mathbf{x}-\mathbf{b}\|\leq\ell
\end{equation}
 is an ellipse with foci points $\mathbf{a},\mathbf{b}$. The set
$\mathbb{X}$ is defined by the inequality
\begin{equation}
\mathbf{f}_{\mathbf{a},\mathbf{b},\ell}(\mathbf{x})\leq0
\end{equation}
where
\begin{equation}
\mathbf{f}_{\mathbf{a},\mathbf{b},\ell}(\mathbf{x})=q_{0}+q_{1}x_{1}+q_{2}x_{2}+q_{3}x_{1}^{2}+q_{4}x_{1}x_{2}+q_{5}x_{2}^{2}
\end{equation}
with
\[
\begin{array}{ccl}
q_{0} & =- & a_{1}^{4}-2a_{1}^{2}a_{2}^{2}+2a_{1}^{2}b_{1}^{2}+2a_{1}^{2}b_{2}^{2}+2a_{1}^{2}\ell^{2}\\
 &  & -a_{2}^{4}+2a_{2}^{2}b_{1}^{2}+2a_{2}^{2}b_{2}^{2}\\
 &  & +2a_{2}^{2}\ell^{2}-b_{1}^{4}-2b_{1}^{2}b_{2}^{2}+2b_{1}^{2}\ell^{2}-b_{2}^{4}+2b_{2}^{2}\ell^{2}-\ell^{4}\\
q_{1} & = & 4a_{1}^{3}-4a_{1}^{2}b_{1}+4a_{1}a_{2}^{2}-4a_{1}b_{1}^{2}-4a_{1}b_{2}^{2}\\
 &  & -4a_{1}\ell^{2}-4a_{2}^{2}b_{1}+4b_{1}^{3}+4b_{1}b_{2}^{2}-4b_{1}\ell^{2}\\
q_{2} & = & 4a_{1}^{2}a_{2}-4a_{1}^{2}b_{2}+4a_{2}^{3}-4a_{2}^{2}b_{2}-4a_{2}b_{1}^{2}\\
 &  & -4a_{2}b_{2}^{2}-4a_{2}\ell^{2}+4b_{1}^{2}b_{2}+4b_{2}^{3}-4b_{2}\ell^{2}\\
q_{3} & = & -4a_{1}^{2}+8a_{1}b_{1}-4b_{1}^{2}+4\ell^{2}\\
q_{4} & = & -8a_{1}a_{2}+8a_{1}b_{2}+8a_{2}b_{1}-8b_{1}b_{2}\\
q_{5} & = & -4a_{2}^{2}+8a_{2}b_{2}-4b_{2}^{2}+4\ell^{2}
\end{array}
\]
 
\end{prop}
\begin{proof}
We have
\begin{equation}
\begin{array}{cl}
 & \|\mathbf{x}-\mathbf{a}\|+\|\mathbf{x}-\mathbf{b}\|=\ell\\
\Leftrightarrow & \sqrt{(x_{1}-a_{1})^{2}+(x_{2}-a_{2})^{2}}+\sqrt{(x_{1}-b_{1})^{2}+(x_{2}-b_{2})^{2}}=\ell
\end{array}
\end{equation}
After some trivial symbolic calculus, we get to get rid of the square
root to get
\begin{equation}
\begin{array}{ccc}
4\left((x_{1}-a_{1})^{2}+(x_{2}-a_{2})^{2}\right)\left((x_{1}-b_{1})^{2}+(x_{2}-b_{2})^{2}\right)\\
-\left(\ell^{2}-(x_{1}-a_{1})^{2}-(x_{2}-a_{2})^{2}-(x_{1}-b_{1})^{2}-(x_{2}-b_{2})^{2}\right)^{2} & = & 0
\end{array}
\end{equation}
We can develop the expression to get the coefficients of the proposition. 
\end{proof}

\subsection{Localization}

We consider an example related to localization which can be seen as
special case of interval data fitting problem \citep{Kreinovich:Shary:16}.
Consider three sonars located at points $\mathbf{a}:(-2,1),\mathbf{b}:(-2,-1),\mathbf{c}:(3,2)$
of the plane. The emitter $\mathbf{a}$ sends a sound which is reflected
by an object at position $\mathbf{x}$ received by $\mathbf{b}$ and
$\mathbf{c}$ (see Figure \ref{ellipseloc0}). From the time of flight
of the sound we want to estimate the position of the object. 

\begin{figure}[H]
\begin{centering}
\includegraphics[width=5cm]{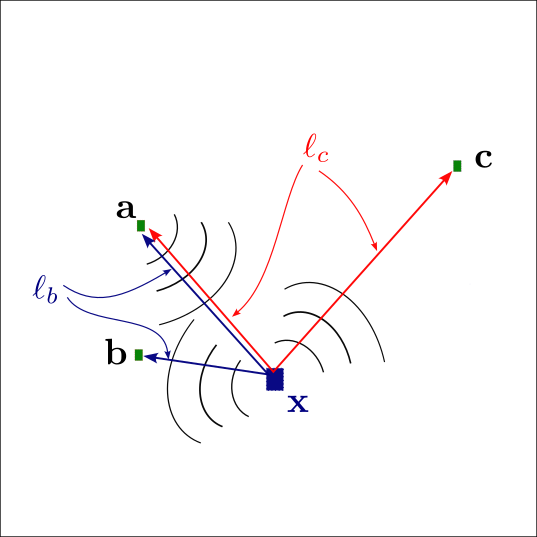}
\par\end{centering}
\caption{The sonar system returns the two distances $\ell_{b}$ and $\ell_{c}$}
\label{ellipseloc0}
\end{figure}

We assume that we were able to collect two distance intervals such
that $\ell_{b}\in[4,6]$ and $\ell_{c}\in[7,9]$. The solution set
$\mathbb{X}$ is defined by 
\begin{equation}
\begin{array}{ccccc}
\text{(i)} & \, & \|\mathbf{x}-\mathbf{a}\|+\|\mathbf{x}-\mathbf{b}\| & = & \ell_{b}\in[4,6]\\
\text{(ii)} &  & \|\mathbf{x}-\mathbf{a}\|+\|\mathbf{x}-\mathbf{c}\| & = & \ell_{c}\in[7,9]
\end{array}\label{eq:inequality:X}
\end{equation}
From Proposition \ref{prop:ellipse:foci}, we get that $\mathbb{X}$
is defined by
\begin{equation}
\mathbb{X}:\left\{ \begin{array}{c}
\mathbf{f}_{\mathbf{a},\mathbf{b},6}(\mathbf{x})\leq0\\
\mathbf{f}_{\mathbf{a},\mathbf{b},4}(\mathbf{x})\geq0\\
\mathbf{f}_{\mathbf{a},\mathbf{c},9}(\mathbf{x})\leq0\\
\mathbf{f}_{\mathbf{a},\mathbf{c},7}(\mathbf{x})\geq0
\end{array}\right.
\end{equation}
Using a paver, we are thus able to get in inner and an outer approximations
for the set of $\mathbb{X}$ (see Figure \ref{fig:ellipseloc3}).The
frame box is $[-7,7]\times[-7,7].$ Figure \ref{ellipseloc1} represents
the inequality (\ref{ellipseloc1},i) and Figure \ref{fig:ellipseloc2}
correspond to the inequality (\ref{ellipseloc1},ii). All results
are guaranteed since outward rounding is implemented \citep{revol2022testing}.

\begin{figure}[H]
\begin{centering}
\includegraphics[width=8cm]{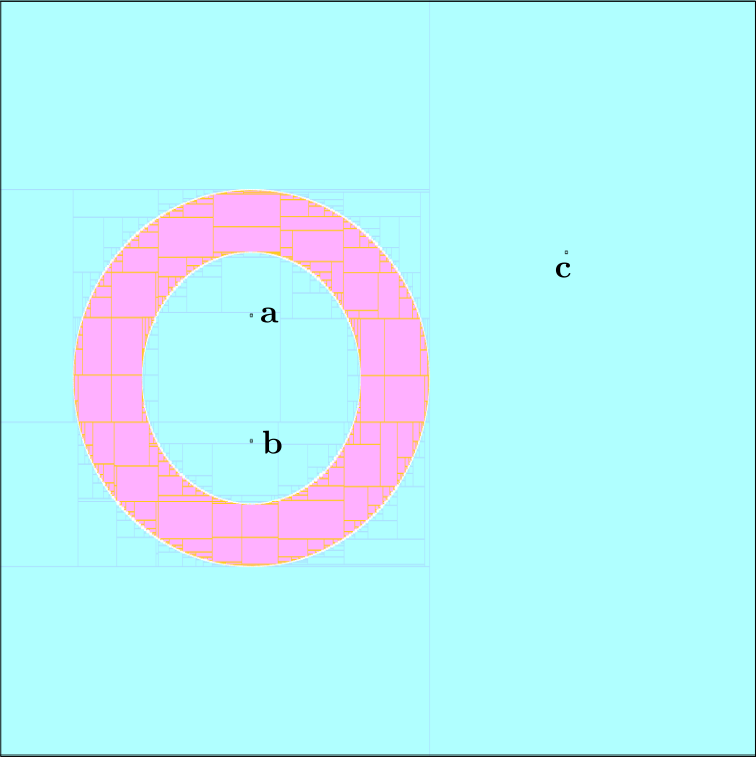}
\par\end{centering}
\caption{Set of positions consistent with the path $\mathbf{a},\mathbf{b}$}
\label{ellipseloc1}
\end{figure}

\begin{figure}[H]
\begin{centering}
\includegraphics[width=8cm]{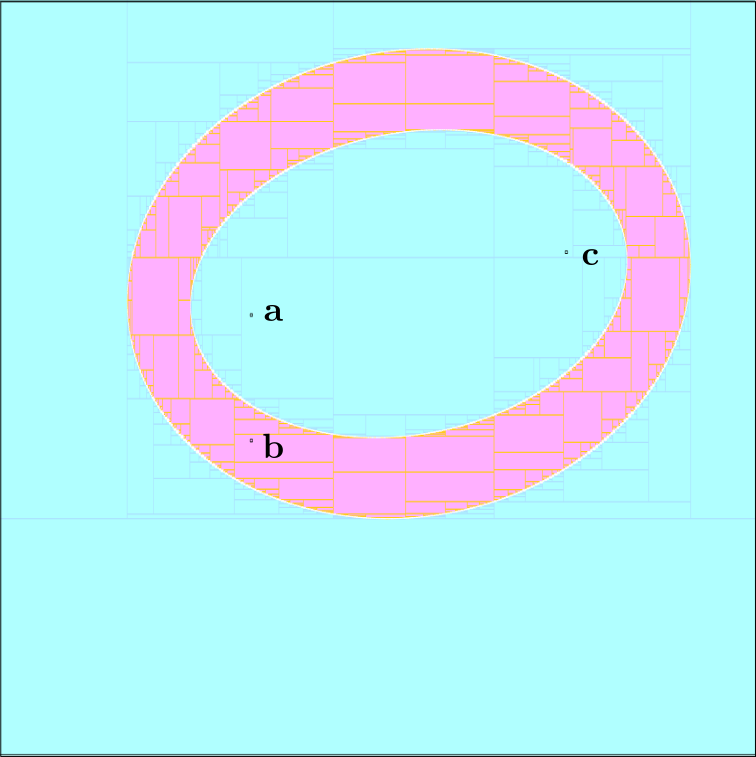}
\par\end{centering}
\caption{Set of positions consistent with the path $\mathbf{a},\mathbf{c}$}
\label{fig:ellipseloc2}
\end{figure}
\begin{figure}[H]
\begin{centering}
\includegraphics[width=8cm]{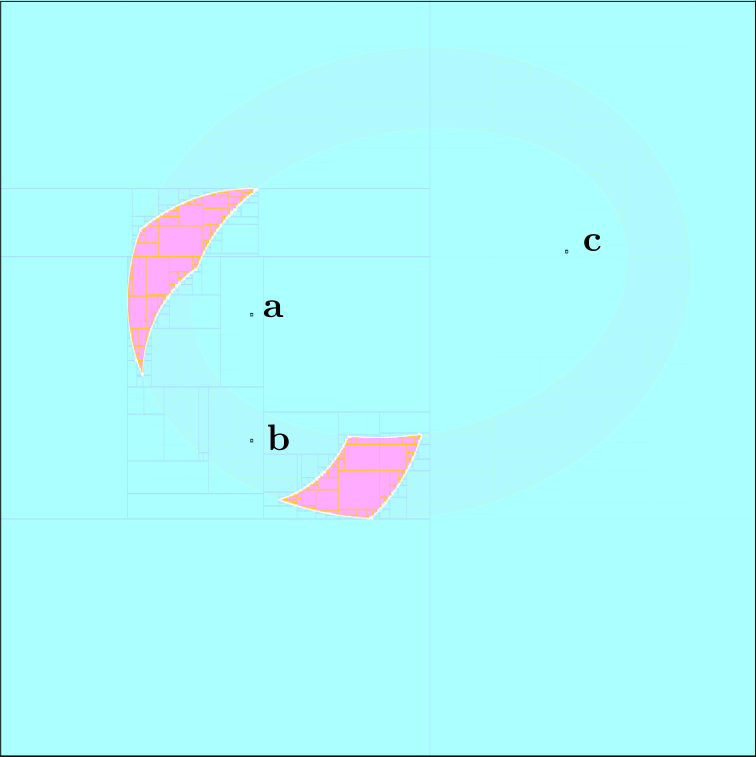}
\par\end{centering}
\caption{Set of positions consistent with the path $\mathbf{a},\mathbf{b}$
and the path $\mathbf{a},\mathbf{c}$}
\label{fig:ellipseloc3}
\end{figure}

\section{Conclusion\label{sec:Conclusion}}

This paper has proposed a minimal contractor and a minimal separator
for an ellipse of the plane. The notion of actions derived from hyperoctahedral
symmetries allowed us to limit the analysis in on part of the constraint
where the monotonicity can be assumed. The symmetries was used to
extend the analysis to the whole plane. 

The goal of this paper was to provide a simple example which illustrates
how to use the hyperoctahedral symmetries in order to build minimal
separators. Now, as shown in \citep{jaulin:quotient:2023}, the use
of these symmetries is more interesting when we deal with projection
problems where quantifier elimination is needed. This type of projection
problem is indeed much more difficult to solve with classical interval
approaches \citep{ratschan}.

The Python code based on Codac \citep{codac} is given in\citep{jaulin:code:ctcellipse:23}.

\bibliographystyle{plain}

\end{document}